\documentclass{amsart}
\usepackage{stmaryrd}
\usepackage{amsmath}
\usepackage{tikz}
\usepackage[hidelinks]{hyperref}
\hypersetup{
  pdftitle={Back Stable Standard Elementary Monomials},
  pdfauthor={Carlos Rodriguez}
}

\title{Back Stable Standard Elementary Monomials}
\author{Carlos Rodriguez}

\newcommand{\Q}{\mathbb{Q}}
\newcommand{\Z}{\mathbb{Z}}

\newcommand{\T}{\mathfrak{T}}

\newcommand{\sem}{\mathbf{e}}
\newcommand{\bsem}[1]{\overleftarrow{\mathbf{e}}_{#1}}
\newcommand{\EE}{\mathbf{E}}
\newcommand{\chm}{\mathbf{h}}
\newcommand{\bchm}[1]{\overleftarrow{\mathbf{h}}_{#1}}
\newcommand{\HH}{\mathbf{H}}
\newcommand{\schub}{\mathfrak{S}}
\newcommand{\bschub}{\overleftarrow{\schub}}
\newcommand{\Rg}{\mathcal{R}}
\newcommand{\RP}{\operatorname{RP}}
\newcommand{\wt}{\operatorname{wt}}
\newcommand{\CC}{\mathcal{C}}

\newcommand{\bE}{\overleftarrow{\mathcal{E}}}
\newcommand{\bH}{\overleftarrow{\mathcal{H}}}
\newcommand{\supp}{\operatorname{supp}}

\newcommand{\Par}{\operatorname{Par}}
\newcommand{\Orb}{\operatorname{Orb}}
\newcommand{\parz}{\operatorname{par}}

\theoremstyle{plain}
\newtheorem{theorem}{Theorem}[section]
\newtheorem{lemma}[theorem]{Lemma}
\newtheorem{corollary}[theorem]{Corollary}
\newtheorem{fact}[theorem]{Fact}
\newtheorem{proposition}[theorem]{Proposition}
\newtheorem{claim}[theorem]{Claim}
\theoremstyle{definition}
\newtheorem{definition}[theorem]{Definition}
\newtheorem{example}[theorem]{Example}

\tikzset{every picture/.style={line width=0.75pt}}

\begin{document}

\begin{abstract}
    We introduce back stable standard elementary monomials and their complete homogeneous analogues. These families form bases of the back stable Schubert ring and record the eventual standard elementary monomial expansions of left-stabilized Schubert polynomials. Dynkin reversal exchanges the two bases while preserving their coefficients. We also show that orbit sums of the back stable coefficients recover the elementary and complete homogeneous coefficients of Stanley symmetric functions. Finally, we study the shifted specialization polynomial
    \[ \T_w(t)=\schub_{1^t\times w}(1), \]
    considered by Fomin and Kirillov, which counts reduced pipe dreams of $1^t\times w$. We express this polynomial in terms of back stable standard elementary monomial coefficients.
\end{abstract}

\maketitle

\section{Introduction}

\subsection{Overview}

Schubert polynomials $\{ \schub_w : w \in S_\infty \}$ form a distinguished basis of $\Q[x_1,x_2, \ldots]$ and admit combinatorial expansions into several bases. One such basis is the standard elementary monomials (SEMs), products of elementary symmetric polynomials in nested sets of variables. Fomin, Gelfand, and Postnikov \cite{FGP} first introduced the connection between Schubert polynomials and SEMs. Since then, these expansions have been studied from several perspectives, including recent work revisiting Schubert polynomial expansions \cite{NadeauSpinkTewari}. Nevertheless, there are several important open questions regarding the SEM expansions of Schubert polynomials
\[ \schub_w = \sum_\alpha \kappa_\alpha^w \cdot \sem_{\alpha}. \]
Among them is the problem of finding a cancellation-free or combinatorial formula for the coefficients $\kappa_\alpha^w$. In this paper, we study these coefficients using back stabilization of Schubert polynomials, an operation devised by Lam, Lee, and Shimozono \cite{LLS}. Given a permutation $w$, consider the Schubert polynomials $\schub_{1^m \times w}$ obtained by adjoining $m$ fixed points to the left of $w$. After shifting the SEMs, the coefficients of $\schub_{1^m \times w}$ in the SEM basis eventually stabilize and have constant support as $m \to \infty$. We formalize this by introducing \emph{back stable standard elementary monomials}, indexed by finitely supported tuples of nonnegative integers on $\Z$, along with their complete homogeneous analogues. 
\begin{theorem}
    The back stable SEMs and the back stable Schubert polynomials form linear bases of the same ambient ring $\Rg$.
\end{theorem}

Expansion in this basis reveals the limiting behavior of $\schub_{1^m \times w}$ as $m \to \infty$.
Back stability reveals symmetries not apparent in the finite setting. One such symmetry is \emph{Dynkin reversal}, which is defined by 
\[ w \mapsto \widehat{w}, \qquad \widehat{w}(i) = 1-w(1-i). \]
For example, if $w$ maps $(1,2,3) \mapsto (3,1,2)$ and is fixed everywhere else, then $\widehat{w}$ maps $(-2,-1,0) \mapsto (-1,0,-2)$. This Dynkin reversal induces a natural symmetry on the back stable Schubert polynomials when considering the back stable SEM expansion.

\begin{theorem}
    The reversing linear map $\bschub_w \mapsto \bschub_{\widehat{w}}$ sends back stable SEMs to CHMs while keeping coefficients unchanged.
\end{theorem}

This allows us to translate every result regarding back stable SEMs into an equivalent result for back stable CHMs. This symmetry is not readily visible in the ordinary SEM setting.

Stanley symmetric functions $F_w$ are also natural to consider in this setting, as they are limits of left-adjoined Schubert polynomials. We prove that the back stable SEM coefficients of $\bschub_w$ refine the coefficients of $F_w$ in the elementary basis. In particular,

\begin{theorem}
  For $w\in S_\Z$ and $\lambda\in\Par$,
  \[ \sum_{\alpha \in \Orb(\lambda)} \overleftarrow{\kappa^w_\alpha} = [e_\lambda]F_w.  \]  
\end{theorem}

Here $\Orb(\lambda)$ is the set of tuples which can be arranged to obtain $\lambda$, and $[e_\lambda]F_w$ is the coefficient of $e_\lambda$ in $F_w$. Thus, the coefficient $[e_\lambda]F_w$ may be obtained by forgetting the positional information in $\CC$ and summing over all back stable SEM coefficients whose indices lie in the same orbit. In this sense, the back stable SEM expansion retains finer information than the elementary expansion of the corresponding Stanley symmetric function. In Section 6, we revisit the \emph{shifted specialization polynomial} $\T_w(t) = \schub_{1^t \times w}(1)$. It counts the reduced pipe dreams of $1^t \times w$. We reformulate this polynomial in terms of the back stable SEM coefficients, and use it to prove a connection between back-stable SEM coefficients and reduced words.

The work is structured as follows. Section 2 gives an overview of SEMs, CHMs, and back stable Schubert polynomials. Section 3 introduces the back stable SEMs and CHMs, proves their respective basis theorems, and shows that they encode stabilization. Section 4 develops Dynkin reversal duality. Section 5 relates the back stable coefficients to Stanley symmetric functions. Section 6 revisits the shifted specialization polynomial from the back stable perspective. Section 7 computes the back stable SEM expansion of certain families of permutations. Section 8 gives a further direction for study.

\subsection{Acknowledgments}

We thank Dora Woodruff for guidance and helpful discussion throughout the project. We also thank Son Nguyen for interesting conversations. 

\section{Background}

\subsection{Schubert Polynomials and Divided Difference Operators}

We define Schubert polynomials using divided difference operators. For any polynomial $f\in\Q[x_1,x_2,\ldots]$, set
\[ \partial_i (f) = \frac{f-s_i f}{x_i-x_{i+1}}, \]
where $s_i$ is the transposition $(i, i+1)$, with the convention that
\[ s_i f(x_1, x_2, \ldots, x_n ) = f(x_1, x_2, \ldots, x_{i+1}, x_{i}, \ldots, x_n ). \]
If $w_0 \in S_n$ is the longest permutation, we define $\schub_{w_0} = x_1^{n-1}x_2^{n-2} \cdots x_{n-1}^1$. The remaining Schubert polynomials are recursively defined by $\partial_i (\schub_w) = \schub_{ws_i}$ if $ws_i <_R w$, where $<_R$ denotes the weak Bruhat order. The resulting family of Schubert polynomials has fundamental algebraic and combinatorial properties.

\begin{theorem}
    $\{ \schub_w : w \in S_\infty \}$ forms a basis of $\Q[x_1,x_2, \ldots]$.
\end{theorem}

This was first shown by Bergeron and Billey \cite{BergeronBilley}. We may also compute $\schub_w$ using the combinatorial model of reduced pipe dreams.

\begin{definition}
    For any $w \in S_n$, a \emph{pipe dream} of $w$ is an $n \times n$ array, where every cell is filled with a \emph{cross} or an \emph{elbow}. For every $i \in [n]$, the wire emitting from the leftmost edge of row $i$ must terminate on the topmost edge of column $w(i)$.     

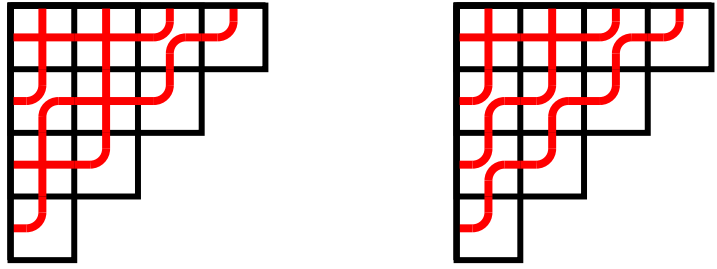
\begin{figure}[ht]
\centering
\begin{tikzpicture}[x=0.6pt,y=0.6pt,yscale=-1,xscale=1]
%uncomment if require: \path (0,300); %set diagram left start at 0, and has height of 300

%Shape: Rectangle [id:dp16026443377124444] 
\draw  [line width=2.25]  (100,90) -- (140,90) -- (140,250) -- (100,250) -- cycle ;
%Shape: Rectangle [id:dp8484218151919661] 
\draw  [line width=2.25]  (100,90) -- (180,90) -- (180,210) -- (100,210) -- cycle ;
%Shape: Rectangle [id:dp18711793759616957] 
\draw  [line width=2.25]  (100,90) -- (220,90) -- (220,170) -- (100,170) -- cycle ;
%Shape: Rectangle [id:dp4897937684903979] 
\draw  [line width=2.25]  (100,90) -- (260,90) -- (260,130) -- (100,130) -- cycle ;
%Straight Lines [id:da021175605456670632] 
\draw [color={rgb, 255:red, 255; green, 0; blue, 0 }  ,draw opacity=1 ][line width=3]    (100,230) -- (110,230) ;
%Straight Lines [id:da9044199649782406] 
\draw [color={rgb, 255:red, 255; green, 0; blue, 0 }  ,draw opacity=1 ][line width=3]    (120,170) -- (120,210) ;
%Straight Lines [id:da2702874674335627] 
\draw [color={rgb, 255:red, 250; green, 0; blue, 0 }  ,draw opacity=1 ][line width=3]    (100,190) -- (140,190) ;
%Shape: Arc [id:dp86216353768473] 
\draw  [draw opacity=0][line width=3]  (120,220) .. controls (120,220) and (120,220) .. (120,220) .. controls (120,220) and (120,220) .. (120,220) .. controls (120,225.52) and (115.52,230) .. (110,230) -- (110,220) -- cycle ; \draw  [color={rgb, 255:red, 255; green, 0; blue, 0 }  ,draw opacity=1 ][line width=3]  (120,220) .. controls (120,220) and (120,220) .. (120,220) .. controls (120,220) and (120,220) .. (120,220) .. controls (120,225.52) and (115.52,230) .. (110,230) ;  
%Straight Lines [id:da43060104670215016] 
\draw [color={rgb, 255:red, 255; green, 0; blue, 0 }  ,draw opacity=1 ][fill={rgb, 255:red, 255; green, 0; blue, 0 }  ,fill opacity=1 ][line width=3]    (120,210) -- (120,220) ;
%Straight Lines [id:da9762531924655603] 
\draw [color={rgb, 255:red, 255; green, 0; blue, 0 }  ,draw opacity=1 ][line width=3]    (180,150) -- (190,150) ;
%Shape: Arc [id:dp7831445612732518] 
\draw  [draw opacity=0][line width=3]  (200,140) .. controls (200,140) and (200,140) .. (200,140) .. controls (200,140) and (200,140) .. (200,140) .. controls (200,145.52) and (195.52,150) .. (190,150) -- (190,140) -- cycle ; \draw  [color={rgb, 255:red, 255; green, 0; blue, 0 }  ,draw opacity=1 ][line width=3]  (200,140) .. controls (200,140) and (200,140) .. (200,140) .. controls (200,140) and (200,140) .. (200,140) .. controls (200,145.52) and (195.52,150) .. (190,150) ;  
%Straight Lines [id:da40564402754228346] 
\draw [color={rgb, 255:red, 255; green, 0; blue, 0 }  ,draw opacity=1 ][fill={rgb, 255:red, 255; green, 0; blue, 0 }  ,fill opacity=1 ][line width=3]    (200,130) -- (200,140) ;
%Straight Lines [id:da9860310871348288] 
\draw [color={rgb, 255:red, 255; green, 0; blue, 0 }  ,draw opacity=1 ][line width=3]    (140,190) -- (150,190) ;
%Shape: Arc [id:dp2560143519567194] 
\draw  [draw opacity=0][line width=3]  (160,180) .. controls (160,180) and (160,180) .. (160,180) .. controls (160,180) and (160,180) .. (160,180) .. controls (160,185.52) and (155.52,190) .. (150,190) -- (150,180) -- cycle ; \draw  [color={rgb, 255:red, 255; green, 0; blue, 0 }  ,draw opacity=1 ][line width=3]  (160,180) .. controls (160,180) and (160,180) .. (160,180) .. controls (160,180) and (160,180) .. (160,180) .. controls (160,185.52) and (155.52,190) .. (150,190) ;  
%Straight Lines [id:da806459264709423] 
\draw [color={rgb, 255:red, 255; green, 0; blue, 0 }  ,draw opacity=1 ][fill={rgb, 255:red, 255; green, 0; blue, 0 }  ,fill opacity=1 ][line width=3]    (160,170) -- (160,180) ;
%Straight Lines [id:da36196287065848376] 
\draw [color={rgb, 255:red, 255; green, 0; blue, 0 }  ,draw opacity=1 ][line width=3]    (220,110) -- (230,110) ;
%Shape: Arc [id:dp4488873880925842] 
\draw  [draw opacity=0][line width=3]  (240,100) .. controls (240,100) and (240,100) .. (240,100) .. controls (240,105.52) and (235.52,110) .. (230,110) -- (230,100) -- cycle ; \draw  [color={rgb, 255:red, 255; green, 0; blue, 0 }  ,draw opacity=1 ][line width=3]  (240,100) .. controls (240,100) and (240,100) .. (240,100) .. controls (240,105.52) and (235.52,110) .. (230,110) ;  
%Straight Lines [id:da7056330034052173] 
\draw [color={rgb, 255:red, 255; green, 0; blue, 0 }  ,draw opacity=1 ][fill={rgb, 255:red, 255; green, 0; blue, 0 }  ,fill opacity=1 ][line width=3]    (240,90) -- (240,100) ;
%Straight Lines [id:da501217301380931] 
\draw [line width=2.25]    (100,90) -- (100,250) ;
%Straight Lines [id:da2339193724542734] 
\draw [line width=2.25]    (100,90) -- (260,90) ;
%Shape: Rectangle [id:dp8240358471912207] 
\draw  [line width=2.25]  (380,90) -- (420,90) -- (420,250) -- (380,250) -- cycle ;
%Shape: Rectangle [id:dp2767455137713142] 
\draw  [line width=2.25]  (380,90) -- (460,90) -- (460,210) -- (380,210) -- cycle ;
%Shape: Rectangle [id:dp2528266253823669] 
\draw  [line width=2.25]  (380,90) -- (500,90) -- (500,170) -- (380,170) -- cycle ;
%Shape: Rectangle [id:dp8589618088527645] 
\draw  [line width=2.25]  (380,90) -- (540,90) -- (540,130) -- (380,130) -- cycle ;
%Straight Lines [id:da9310565557365451] 
\draw [color={rgb, 255:red, 255; green, 0; blue, 0 }  ,draw opacity=1 ][line width=3]    (380,230) -- (390,230) ;
%Shape: Arc [id:dp34914051234589016] 
\draw  [draw opacity=0][line width=3]  (400,220) .. controls (400,220) and (400,220) .. (400,220) .. controls (400,220) and (400,220) .. (400,220) .. controls (400,225.52) and (395.52,230) .. (390,230) -- (390,220) -- cycle ; \draw  [color={rgb, 255:red, 255; green, 0; blue, 0 }  ,draw opacity=1 ][line width=3]  (400,220) .. controls (400,220) and (400,220) .. (400,220) .. controls (400,220) and (400,220) .. (400,220) .. controls (400,225.52) and (395.52,230) .. (390,230) ;  
%Straight Lines [id:da743710849693085] 
\draw [color={rgb, 255:red, 255; green, 0; blue, 0 }  ,draw opacity=1 ][fill={rgb, 255:red, 255; green, 0; blue, 0 }  ,fill opacity=1 ][line width=3]    (400,210) -- (400,220) ;
%Straight Lines [id:da3784741417313435] 
\draw [color={rgb, 255:red, 255; green, 0; blue, 0 }  ,draw opacity=1 ][line width=3]    (460,150) -- (470,150) ;
%Shape: Arc [id:dp8116841398649731] 
\draw  [draw opacity=0][line width=3]  (480,140) .. controls (480,140) and (480,140) .. (480,140) .. controls (480,140) and (480,140) .. (480,140) .. controls (480,145.52) and (475.52,150) .. (470,150) -- (470,140) -- cycle ; \draw  [color={rgb, 255:red, 255; green, 0; blue, 0 }  ,draw opacity=1 ][line width=3]  (480,140) .. controls (480,140) and (480,140) .. (480,140) .. controls (480,140) and (480,140) .. (480,140) .. controls (480,145.52) and (475.52,150) .. (470,150) ;  
%Straight Lines [id:da5226382579236267] 
\draw [color={rgb, 255:red, 255; green, 0; blue, 0 }  ,draw opacity=1 ][fill={rgb, 255:red, 255; green, 0; blue, 0 }  ,fill opacity=1 ][line width=3]    (480,130) -- (480,140) ;
%Straight Lines [id:da06970196238130288] 
\draw [color={rgb, 255:red, 255; green, 0; blue, 0 }  ,draw opacity=1 ][line width=3]    (420,190) -- (430,190) ;
%Shape: Arc [id:dp22854897835470378] 
\draw  [draw opacity=0][line width=3]  (440,180) .. controls (440,180) and (440,180) .. (440,180) .. controls (440,180) and (440,180) .. (440,180) .. controls (440,185.52) and (435.52,190) .. (430,190) -- (430,180) -- cycle ; \draw  [color={rgb, 255:red, 255; green, 0; blue, 0 }  ,draw opacity=1 ][line width=3]  (440,180) .. controls (440,180) and (440,180) .. (440,180) .. controls (440,180) and (440,180) .. (440,180) .. controls (440,185.52) and (435.52,190) .. (430,190) ;  
%Straight Lines [id:da5911699353029283] 
\draw [color={rgb, 255:red, 255; green, 0; blue, 0 }  ,draw opacity=1 ][fill={rgb, 255:red, 255; green, 0; blue, 0 }  ,fill opacity=1 ][line width=3]    (440,170) -- (440,180) ;
%Straight Lines [id:da440399373853964] 
\draw [color={rgb, 255:red, 255; green, 0; blue, 0 }  ,draw opacity=1 ][line width=3]    (500,110) -- (510,110) ;
%Shape: Arc [id:dp13747435713404044] 
\draw  [draw opacity=0][line width=3]  (520,100) .. controls (520,100) and (520,100) .. (520,100) .. controls (520,105.52) and (515.52,110) .. (510,110) -- (510,100) -- cycle ; \draw  [color={rgb, 255:red, 255; green, 0; blue, 0 }  ,draw opacity=1 ][line width=3]  (520,100) .. controls (520,100) and (520,100) .. (520,100) .. controls (520,105.52) and (515.52,110) .. (510,110) ;  
%Straight Lines [id:da6584309724333856] 
\draw [color={rgb, 255:red, 255; green, 0; blue, 0 }  ,draw opacity=1 ][fill={rgb, 255:red, 255; green, 0; blue, 0 }  ,fill opacity=1 ][line width=3]    (520,90) -- (520,100) ;
%Straight Lines [id:da7427281695135726] 
\draw [line width=2.25]    (380,90) -- (380,250) ;
%Straight Lines [id:da5662062066972827] 
\draw [line width=2.25]    (380,90) -- (540,90) ;
%Straight Lines [id:da8085467258945863] 
\draw [color={rgb, 255:red, 250; green, 0; blue, 0 }  ,draw opacity=1 ][line width=3]    (140,150) -- (180,150) ;
%Straight Lines [id:da33070619553817027] 
\draw [color={rgb, 255:red, 255; green, 0; blue, 0 }  ,draw opacity=1 ][line width=3]    (160,130) -- (160,170) ;
%Straight Lines [id:da20113756578220499] 
\draw [color={rgb, 255:red, 250; green, 0; blue, 0 }  ,draw opacity=1 ][line width=3]    (100,110) -- (140,110) ;
%Straight Lines [id:da18367366566987364] 
\draw [color={rgb, 255:red, 255; green, 0; blue, 0 }  ,draw opacity=1 ][line width=3]    (120,90) -- (120,130) ;
%Straight Lines [id:da8266934396035321] 
\draw [color={rgb, 255:red, 255; green, 0; blue, 0 }  ,draw opacity=1 ][line width=3]    (100,150) -- (110,150) ;
%Shape: Arc [id:dp7167801784804237] 
\draw  [draw opacity=0][line width=3]  (120,140) .. controls (120,140) and (120,140) .. (120,140) .. controls (120,140) and (120,140) .. (120,140) .. controls (120,145.52) and (115.52,150) .. (110,150) -- (110,140) -- cycle ; \draw  [color={rgb, 255:red, 255; green, 0; blue, 0 }  ,draw opacity=1 ][line width=3]  (120,140) .. controls (120,140) and (120,140) .. (120,140) .. controls (120,140) and (120,140) .. (120,140) .. controls (120,145.52) and (115.52,150) .. (110,150) ;  
%Straight Lines [id:da023109141327227745] 
\draw [color={rgb, 255:red, 255; green, 0; blue, 0 }  ,draw opacity=1 ][fill={rgb, 255:red, 255; green, 0; blue, 0 }  ,fill opacity=1 ][line width=3]    (120,130) -- (120,140) ;
%Straight Lines [id:da40223385767337316] 
\draw [color={rgb, 255:red, 255; green, 0; blue, 0 }  ,draw opacity=1 ][line width=3]    (140,150) -- (130,150) ;
%Shape: Arc [id:dp6719884376977818] 
\draw  [draw opacity=0][line width=3]  (120,160) .. controls (120,160) and (120,160) .. (120,160) .. controls (120,160) and (120,160) .. (120,160) .. controls (120,154.48) and (124.48,150) .. (130,150) -- (130,160) -- cycle ; \draw  [color={rgb, 255:red, 255; green, 0; blue, 0 }  ,draw opacity=1 ][line width=3]  (120,160) .. controls (120,160) and (120,160) .. (120,160) .. controls (120,160) and (120,160) .. (120,160) .. controls (120,154.48) and (124.48,150) .. (130,150) ;  
%Straight Lines [id:da42288587896398455] 
\draw [color={rgb, 255:red, 255; green, 0; blue, 0 }  ,draw opacity=1 ][fill={rgb, 255:red, 255; green, 0; blue, 0 }  ,fill opacity=1 ][line width=3]    (120,170) -- (120,160) ;
%Straight Lines [id:da23187439203482352] 
\draw [color={rgb, 255:red, 255; green, 0; blue, 0 }  ,draw opacity=1 ][line width=3]    (180,110) -- (190,110) ;
%Shape: Arc [id:dp9294367944666487] 
\draw  [draw opacity=0][line width=3]  (200,100) .. controls (200,100) and (200,100) .. (200,100) .. controls (200,105.52) and (195.52,110) .. (190,110) -- (190,100) -- cycle ; \draw  [color={rgb, 255:red, 255; green, 0; blue, 0 }  ,draw opacity=1 ][line width=3]  (200,100) .. controls (200,100) and (200,100) .. (200,100) .. controls (200,105.52) and (195.52,110) .. (190,110) ;  
%Straight Lines [id:da05009222454786233] 
\draw [color={rgb, 255:red, 255; green, 0; blue, 0 }  ,draw opacity=1 ][fill={rgb, 255:red, 255; green, 0; blue, 0 }  ,fill opacity=1 ][line width=3]    (200,90) -- (200,100) ;
%Straight Lines [id:da5387225610764496] 
\draw [color={rgb, 255:red, 255; green, 0; blue, 0 }  ,draw opacity=1 ][line width=3]    (220,110) -- (210,110) ;
%Shape: Arc [id:dp7724210396814105] 
\draw  [draw opacity=0][line width=3]  (200,120) .. controls (200,120) and (200,120) .. (200,120) .. controls (200,114.48) and (204.48,110) .. (210,110) -- (210,120) -- cycle ; \draw  [color={rgb, 255:red, 255; green, 0; blue, 0 }  ,draw opacity=1 ][line width=3]  (200,120) .. controls (200,120) and (200,120) .. (200,120) .. controls (200,114.48) and (204.48,110) .. (210,110) ;  
%Straight Lines [id:da40347824445968794] 
\draw [color={rgb, 255:red, 255; green, 0; blue, 0 }  ,draw opacity=1 ][fill={rgb, 255:red, 255; green, 0; blue, 0 }  ,fill opacity=1 ][line width=3]    (200,130) -- (200,120) ;
%Straight Lines [id:da5430327016658052] 
\draw [color={rgb, 255:red, 250; green, 0; blue, 0 }  ,draw opacity=1 ][line width=3]    (140,110) -- (180,110) ;
%Straight Lines [id:da586594127646262] 
\draw [color={rgb, 255:red, 255; green, 0; blue, 0 }  ,draw opacity=1 ][line width=3]    (160,90) -- (160,130) ;
%Straight Lines [id:da22614592752138685] 
\draw [color={rgb, 255:red, 250; green, 0; blue, 0 }  ,draw opacity=1 ][line width=3]    (380,110) -- (420,110) ;
%Straight Lines [id:da7959183535295339] 
\draw [color={rgb, 255:red, 255; green, 0; blue, 0 }  ,draw opacity=1 ][line width=3]    (400,90) -- (400,130) ;
%Straight Lines [id:da4252416056142292] 
\draw [color={rgb, 255:red, 250; green, 0; blue, 0 }  ,draw opacity=1 ][line width=3]    (420,110) -- (460,110) ;
%Straight Lines [id:da4409494538120303] 
\draw [color={rgb, 255:red, 255; green, 0; blue, 0 }  ,draw opacity=1 ][line width=3]    (440,90) -- (440,130) ;
%Straight Lines [id:da4827601657160071] 
\draw [color={rgb, 255:red, 255; green, 0; blue, 0 }  ,draw opacity=1 ][line width=3]    (380,190) -- (390,190) ;
%Shape: Arc [id:dp3552476756762152] 
\draw  [draw opacity=0][line width=3]  (400,180) .. controls (400,180) and (400,180) .. (400,180) .. controls (400,180) and (400,180) .. (400,180) .. controls (400,185.52) and (395.52,190) .. (390,190) -- (390,180) -- cycle ; \draw  [color={rgb, 255:red, 255; green, 0; blue, 0 }  ,draw opacity=1 ][line width=3]  (400,180) .. controls (400,180) and (400,180) .. (400,180) .. controls (400,180) and (400,180) .. (400,180) .. controls (400,185.52) and (395.52,190) .. (390,190) ;  
%Straight Lines [id:da2752519752771184] 
\draw [color={rgb, 255:red, 255; green, 0; blue, 0 }  ,draw opacity=1 ][fill={rgb, 255:red, 255; green, 0; blue, 0 }  ,fill opacity=1 ][line width=3]    (400,170) -- (400,180) ;
%Straight Lines [id:da41432543089875296] 
\draw [color={rgb, 255:red, 255; green, 0; blue, 0 }  ,draw opacity=1 ][line width=3]    (420,190) -- (410,190) ;
%Shape: Arc [id:dp9973652225613536] 
\draw  [draw opacity=0][line width=3]  (400,200) .. controls (400,200) and (400,200) .. (400,200) .. controls (400,200) and (400,200) .. (400,200) .. controls (400,194.48) and (404.48,190) .. (410,190) -- (410,200) -- cycle ; \draw  [color={rgb, 255:red, 255; green, 0; blue, 0 }  ,draw opacity=1 ][line width=3]  (400,200) .. controls (400,200) and (400,200) .. (400,200) .. controls (400,200) and (400,200) .. (400,200) .. controls (400,194.48) and (404.48,190) .. (410,190) ;  
%Straight Lines [id:da08757286070096726] 
\draw [color={rgb, 255:red, 255; green, 0; blue, 0 }  ,draw opacity=1 ][fill={rgb, 255:red, 255; green, 0; blue, 0 }  ,fill opacity=1 ][line width=3]    (400,210) -- (400,200) ;
%Straight Lines [id:da7868285832380267] 
\draw [color={rgb, 255:red, 255; green, 0; blue, 0 }  ,draw opacity=1 ][line width=3]    (380,150) -- (390,150) ;
%Shape: Arc [id:dp14648492121101986] 
\draw  [draw opacity=0][line width=3]  (400,140) .. controls (400,140) and (400,140) .. (400,140) .. controls (400,140) and (400,140) .. (400,140) .. controls (400,145.52) and (395.52,150) .. (390,150) -- (390,140) -- cycle ; \draw  [color={rgb, 255:red, 255; green, 0; blue, 0 }  ,draw opacity=1 ][line width=3]  (400,140) .. controls (400,140) and (400,140) .. (400,140) .. controls (400,140) and (400,140) .. (400,140) .. controls (400,145.52) and (395.52,150) .. (390,150) ;  
%Straight Lines [id:da7625718270779194] 
\draw [color={rgb, 255:red, 255; green, 0; blue, 0 }  ,draw opacity=1 ][fill={rgb, 255:red, 255; green, 0; blue, 0 }  ,fill opacity=1 ][line width=3]    (400,130) -- (400,140) ;
%Straight Lines [id:da9917417542930363] 
\draw [color={rgb, 255:red, 255; green, 0; blue, 0 }  ,draw opacity=1 ][line width=3]    (420,150) -- (410,150) ;
%Shape: Arc [id:dp09405753668614614] 
\draw  [draw opacity=0][line width=3]  (400,160) .. controls (400,160) and (400,160) .. (400,160) .. controls (400,160) and (400,160) .. (400,160) .. controls (400,154.48) and (404.48,150) .. (410,150) -- (410,160) -- cycle ; \draw  [color={rgb, 255:red, 255; green, 0; blue, 0 }  ,draw opacity=1 ][line width=3]  (400,160) .. controls (400,160) and (400,160) .. (400,160) .. controls (400,160) and (400,160) .. (400,160) .. controls (400,154.48) and (404.48,150) .. (410,150) ;  
%Straight Lines [id:da05159923276451983] 
\draw [color={rgb, 255:red, 255; green, 0; blue, 0 }  ,draw opacity=1 ][fill={rgb, 255:red, 255; green, 0; blue, 0 }  ,fill opacity=1 ][line width=3]    (400,170) -- (400,160) ;
%Straight Lines [id:da0689803504146439] 
\draw [color={rgb, 255:red, 255; green, 0; blue, 0 }  ,draw opacity=1 ][line width=3]    (420,150) -- (430,150) ;
%Shape: Arc [id:dp331092310378841] 
\draw  [draw opacity=0][line width=3]  (440,140) .. controls (440,140) and (440,140) .. (440,140) .. controls (440,140) and (440,140) .. (440,140) .. controls (440,145.52) and (435.52,150) .. (430,150) -- (430,140) -- cycle ; \draw  [color={rgb, 255:red, 255; green, 0; blue, 0 }  ,draw opacity=1 ][line width=3]  (440,140) .. controls (440,140) and (440,140) .. (440,140) .. controls (440,140) and (440,140) .. (440,140) .. controls (440,145.52) and (435.52,150) .. (430,150) ;  
%Straight Lines [id:da11144613311689366] 
\draw [color={rgb, 255:red, 255; green, 0; blue, 0 }  ,draw opacity=1 ][fill={rgb, 255:red, 255; green, 0; blue, 0 }  ,fill opacity=1 ][line width=3]    (440,130) -- (440,140) ;
%Straight Lines [id:da4274300765308937] 
\draw [color={rgb, 255:red, 255; green, 0; blue, 0 }  ,draw opacity=1 ][line width=3]    (460,150) -- (450,150) ;
%Shape: Arc [id:dp520891191245625] 
\draw  [draw opacity=0][line width=3]  (440,160) .. controls (440,160) and (440,160) .. (440,160) .. controls (440,160) and (440,160) .. (440,160) .. controls (440,154.48) and (444.48,150) .. (450,150) -- (450,160) -- cycle ; \draw  [color={rgb, 255:red, 255; green, 0; blue, 0 }  ,draw opacity=1 ][line width=3]  (440,160) .. controls (440,160) and (440,160) .. (440,160) .. controls (440,160) and (440,160) .. (440,160) .. controls (440,154.48) and (444.48,150) .. (450,150) ;  
%Straight Lines [id:da6679303396074775] 
\draw [color={rgb, 255:red, 255; green, 0; blue, 0 }  ,draw opacity=1 ][fill={rgb, 255:red, 255; green, 0; blue, 0 }  ,fill opacity=1 ][line width=3]    (440,170) -- (440,160) ;
%Straight Lines [id:da8533933341033813] 
\draw [color={rgb, 255:red, 255; green, 0; blue, 0 }  ,draw opacity=1 ][line width=3]    (460,110) -- (470,110) ;
%Shape: Arc [id:dp35138772164279086] 
\draw  [draw opacity=0][line width=3]  (480,100) .. controls (480,100) and (480,100) .. (480,100) .. controls (480,105.52) and (475.52,110) .. (470,110) -- (470,100) -- cycle ; \draw  [color={rgb, 255:red, 255; green, 0; blue, 0 }  ,draw opacity=1 ][line width=3]  (480,100) .. controls (480,100) and (480,100) .. (480,100) .. controls (480,105.52) and (475.52,110) .. (470,110) ;  
%Straight Lines [id:da3472201703688663] 
\draw [color={rgb, 255:red, 255; green, 0; blue, 0 }  ,draw opacity=1 ][fill={rgb, 255:red, 255; green, 0; blue, 0 }  ,fill opacity=1 ][line width=3]    (480,90) -- (480,100) ;
%Straight Lines [id:da11681640230766077] 
\draw [color={rgb, 255:red, 255; green, 0; blue, 0 }  ,draw opacity=1 ][line width=3]    (500,110) -- (490,110) ;
%Shape: Arc [id:dp6782580796342749] 
\draw  [draw opacity=0][line width=3]  (480,120) .. controls (480,120) and (480,120) .. (480,120) .. controls (480,114.48) and (484.48,110) .. (490,110) -- (490,120) -- cycle ; \draw  [color={rgb, 255:red, 255; green, 0; blue, 0 }  ,draw opacity=1 ][line width=3]  (480,120) .. controls (480,120) and (480,120) .. (480,120) .. controls (480,114.48) and (484.48,110) .. (490,110) ;  
%Straight Lines [id:da550321730153722] 
\draw [color={rgb, 255:red, 255; green, 0; blue, 0 }  ,draw opacity=1 ][fill={rgb, 255:red, 255; green, 0; blue, 0 }  ,fill opacity=1 ][line width=3]    (480,130) -- (480,120) ;
%Straight Lines [id:da2565451348098954] 
\draw [line width=2.25]    (100,90) -- (100,250) ;
%Straight Lines [id:da17709214313610855] 
\draw [line width=2.25]    (100,90) -- (260,90) ;
%Straight Lines [id:da316482204568055] 
\draw [line width=2.25]    (380,90) -- (540,90) ;
%Straight Lines [id:da7959579449792387] 
\draw [line width=2.25]    (380,90) -- (380,250) ;
\end{tikzpicture}
\caption{Two pipe dreams for $3124$.}
\end{figure}

    A pipe dream is \emph{reduced} if no pair of wires cross more than once. In the above figure, the right pipe dream is reduced while the left is not. For any $w \in S_\infty$ let $\RP(w)$ be the set of reduced pipe dreams of $w$.
\end{definition}

The computation of Schubert polynomials comes from the following fundamental theorem.

\begin{theorem}[{\cite[1.1]{BJS}}]
    Let $w\in S_n$. For any $D \in \RP(w)$, let $\wt(D)$ be the tuple whose $i$th entry is the number of crosses in the $i$th row of $D$. Then
    \[ \schub_w = \sum_{D \in \RP(w)} x^{\wt(D)}. \]
\end{theorem}

Note that this polynomial is homogeneous of degree $\ell(w)$, as the number of crosses in each reduced pipe dream is precisely $\ell(w)$.

\begin{definition}
    The \emph{Lehmer code} of $w$ is the sequence
    \[ L(w)_i=\#\{j>i:w(j)<w(i)\}. \]
\end{definition}

The lexicographically largest monomial in $\schub_w$ is $x_1^{L(w)_1}x_2^{L(w)_2} \cdots x_n^{L(w)_n}$, which corresponds to the pipe dream in which the first $L(w)_i$ cells of row $i \in [n]$ are filled with crosses and all other cells are filled with elbows. We refer to this pipe dream as the bottom pipe dream.

\subsection{SEMs and CHMs} 

\begin{definition}
    For nonnegative $d$ and positive $k$, let $\sem_d^k=\sem_d(x_1,\ldots,x_k)$ denote the degree-$d$ elementary symmetric polynomial in the first $k$ variables, with $\sem_0^k=1$. A \emph{standard elementary monomial} is a product
    \[ \sem_{\vec a} = \sem_{a_1}^1 \sem_{a_2}^2 \sem_{a_3}^3 \cdots\]
    where $a_i \le i$ for all $i$, and only finitely many of the $a_i$'s are non-zero.
\end{definition}

\begin{theorem}[{\cite[Theorem 3.3]{FGP}}]
    $\{ \sem_{\vec a} \}$ forms a basis of $\Q[x_1,x_2, \ldots]$.
\end{theorem}

Moreover, they showed that the expansion of Schubert polynomials into the SEM basis yields rich combinatorial structure through its connection to quantum Schubert polynomials. Since then, there has been substantial work aimed at understanding the coefficients of the expansion
\[ \schub_w = \sum_{\vec a} \kappa_{\vec a}^w \cdot \sem_{\vec a}. \] 

Winkel \cite{Winkel} gave a determinantal formula for the SEM expansion of Schur polynomials, and observed that the coefficients $\kappa_{\vec a}^w$ tended to be small in absolute value. Hatam, Johnson, Liu, and Macaulay \cite{HJLM} gave a determinantal formula for $\schub_w$ when $w$ avoids a certain set of $13$ patterns. Woodruff \cite{Woodruff} characterized all $w$ for which $\schub_w = \sem_{\vec a}$ for some $\vec a$, which is equivalent to $\kappa_{\vec a}^w \neq 0$ for a unique choice of $\vec a$. A related family of polynomials is given as follows:

\begin{definition}
    For nonnegative $d$ and positive $k$, let $\chm_d^k=\chm_d(x_1,\ldots,x_k)$ denote the degree-$d$ complete homogeneous symmetric polynomial in the first $k$ variables, with $\chm_0^k=1$. A \emph{complete homogeneous monomial} is a product
    \[ \chm_{\vec a} = \chm_{a_1}^1 \chm_{a_2}^2 \chm_{a_3}^3 \cdots\]
    where the $a_i$ are nonnegative integers and only finitely many are nonzero.
\end{definition}

\begin{proposition}
    $\{ \chm_{\vec a} \}$ forms a basis of $\Q[x_1,x_2, \ldots]$.
\end{proposition}

\begin{proof}
    Fix $n\geq 1$ and order monomials in $\Q[x_1,\ldots,x_n]$ lexicographically with $x_n>\cdots>x_1$. The leading monomial of $\chm_d(x_1,\ldots,x_i)$ is $x_i^d$. Hence the leading monomial of
    \[ \chm_{a_1}^1\chm_{a_2}^2\cdots\chm_{a_n}^n \]
    is $x_1^{a_1}x_2^{a_2}\cdots x_n^{a_n}$. These leading monomials are distinct and exhaust the monomial basis, so the transition matrix is unitriangular. Letting $n$ vary proves the result.
\end{proof}

The CHMs are less studied than the SEMs, but they arise naturally alongside them in the work of Fomin, Gelfand, and Postnikov. In particular, their quantization map sends complete homogeneous symmetric polynomials to determinants of quantum elementary symmetric polynomials, and more generally sends CHMs to products of such determinants. Thus, like SEMs, CHMs admit a particularly tractable description under quantization, motivating their consideration here.

\subsection{Back Stable Schubert Polynomials}

Following Lam, Lee, and Shimozono, let $S_\Z$ denote the permutations of $\Z$ with finite support. Let $\gamma$ be the shift
\[ \gamma(s_i) = s_{i+1}, \qquad \gamma(x_i) = x_{i+1}.\] 
Since every $w \in S_\Z$ has finite support, choose an interval $[p,q]$ for which $w(i)=i$ for all $i \not \in [p,q]$. Define
\[ \schub_{w}^{[p,q]} = \gamma^{p-1} \schub_{\gamma^{1-p}(w)} \]
which may be regarded as the ordinary Schubert polynomial of $w$ after translating its variables to the interval $[p,q]$. The back stable Schubert polynomial is the coefficientwise limit
\[ \bschub_w =\lim_{\substack{p \to -\infty \\ q \to \infty}} \schub^{[p,q]}_w. \]
Since every $w \in S_\Z$ has finite support, for all sufficiently large $m$ the shift $\gamma^m(w) \in S_\infty$. Thus, up to shifting, it usually suffices to only consider $w \in S_\infty$. In these cases, we can use the more natural formulation
\[ \bschub_w = \lim_{m \to \infty} \gamma^{-m}\schub_{1^m \times w}. \]

\begin{example}
Take the simple case of $w = 312$. Note that
\[ \schub_{1^m \times 312} = \sum_{1 \le i \le j \le m+1} x_i x_j. \]
Thus,
\begin{align*}
\bschub_{312} &= \lim_{m \to \infty} \gamma^{-m} \schub_{1^m \times 312} \\
&= \lim_{m \to \infty} \sum_{1-m \le i \le j \le 1} x_i x_j \\
&= \sum_{i \le j \le 1} x_i x_j.
\end{align*}
\end{example}

Evidently, the limit need not lie in $\Q[x_i\mid i \in \Z]$. Lam, Lee, and Shimozono proved that back stable Schubert polynomials lie in the ring $\Rg = \Lambda^- \otimes \Q[x_i\mid i \in \Z]$, where $\Lambda^- = \Lambda(\ldots, x_{-1}, x_0)$ is the ring of symmetric functions on nonpositively indexed variables. The tensor-product notation distinguishes the symmetric-function factor from the polynomial factor. Informally, the potentially infinite dependence on variables $(x_i)_{i \le 0}$ occurs symmetrically. All limits in what follows are coefficientwise limits in this ring. The following fundamental theorem makes $\Rg$ a natural choice for the ambient ring.

\begin{theorem}[Lam--Lee--Shimozono {\cite[3.5]{LLS}}]
    The set $\left \{ \bschub_w : w \in S_\Z \right \}$ forms a basis of $\Rg$.
\end{theorem}

Thus, the back stable setting retains the structure of ordinary
Schubert theory. By passing to the back stable limit, we can study the eventual
behavior of these coefficients in a setting where stabilization is built into the
algebraic structure.

\section{Back Stable SEMs and CHMs}

\subsection{Constructing Back Stable SEMs and CHMs}

To see how an SEM lends itself to back stabilization, it will be useful to look at a single elementary symmetric polynomial first. Consider
\[ \sem_d^k = \sem_d(x_1, x_2, \ldots, x_k). \]
Back stabilization should introduce infinitely many symmetric variables extending to the left. More precisely, we should obtain the back stable analogue by letting the left endpoint tend to $-\infty$.
\[ \lim_{p \to -\infty} \sem_d(x_p, x_{p+1}, \ldots, x_k) = \sem_d(\ldots, x_{k-1}, x_k). \]
This is well defined in the back stable ring; it is the usual elementary symmetric function in $x_k,x_{k-1}, \ldots$. Moreover, there is no reason for $k$ to be positive. Thus, for any $k \in \Z$ we may regard $\sem_d(\ldots, x_{k-1},x_k)$ as the back stable analogue of $\sem_d^k$. Back stable SEMs, in turn, will be finite products of $\sem_d(\ldots, x_{k-1},x_k)$ whose right endpoints $k$ are pairwise distinct. Let $\CC$ denote the monoid of finitely supported tuples of nonnegative integers indexed by $\Z$, with componentwise addition. Then back stable SEMs may be indexed by elements of $\CC$.

For $\alpha\in\CC$, we write
\[ |\alpha|=\sum_{i\in\Z}\alpha_i, \qquad
   \gamma(\alpha)_i=\alpha_{i-1}, \]
and let $\parz(\alpha)$ be the partition obtained by sorting the nonzero entries of $\alpha$ in weakly decreasing order.

\begin{definition}
    For any $\alpha \in \CC$,
    \[ \bsem{\alpha} = \prod_{i \in \Z} \sem_{\alpha_i}(\ldots,x_{i-1},x_i) \]
    is the \emph{back stable SEM} with index $\alpha$.
\end{definition}

Note that this definition immediately implies $\bsem{\alpha} \in \Rg$, as it is a finite product of elements in $\Rg$. The same reasoning holds for CHMs.

\begin{definition}
    For any $\alpha \in \CC$,
    \[ \bchm{\alpha} = \prod_{i \in \Z} \chm_{\alpha_i}(\ldots, x_{i-1}, x_i) \]
    is the \emph{back stable CHM} indexed by $\alpha$.
\end{definition}

Likewise, $\bchm{\alpha} \in \Rg$ as it is the finite product of elements in $\Rg$. This allows us to describe the families
\[ \bE = \left \{ \bsem{\alpha} : \alpha \in \CC \right \}, \qquad \bH = \left \{ \bchm{\alpha} : \alpha \in \CC \right \}. \]
In particular, the basis theorems of ordinary SEMs and CHMs translate into $\Rg$.

\subsection{Back Stable SEM and CHM Basis Theorems}

To prove the analogous basis theorem, we use truncation at a sufficiently negative index for linear independence and the Fomin--Gelfand--Postnikov straightening argument for spanning.

\begin{theorem}
    The set $\bE$ forms a linear $\Q$-basis of $\Rg$.
\end{theorem}

\begin{proof}
    For spanning, define $p\in\Z$ by
    \[ \rho_p : \Rg \longrightarrow \Q[x_p,x_{p+1},\ldots] \]
    being the truncation which sets $x_i=0$ for all $i<p$. If $p$ lies below the support of $\alpha$, then
    \[ \rho_p(\bsem{\alpha})
    =\prod_{i\in\Z}\sem_{\alpha_i}(x_p,x_{p+1},\ldots,x_i). \]
    For sufficiently small $p$, this is an admissible ordinary SEM.

    We begin with linear independence. Suppose
    \[ \sum_{\alpha\in A}c_\alpha \cdot\bsem{\alpha}=0 \]
    for some finite $A\subseteq\CC$. Choose $p\ll0$ such that every $\rho_p(\bsem{\alpha})$, $\alpha\in A$, is admissible and these truncations are all distinct. Applying $\rho_p$ gives a relation among distinct ordinary SEMs. Since the ordinary SEMs form a basis, $c_\alpha=0$ for every $\alpha\in A$.

    For spanning, we use the Fomin--Gelfand--Postnikov straightening argument \cite[3.3]{FGP}. Let $f \in \Rg$ be homogeneous of degree $d$. Since $f$ is symmetric sufficiently far to the left and only uses variables bounded above, choose $b,q\in\Z$ such that
    \[ f\in\Lambda(\ldots,x_{b-1},x_b)\otimes
    \Q[x_{b+1},\ldots,x_q]. \]
    The symmetric-function factor is generated by the $\sem_r(\ldots,x_b)$, while, for $i>b$,
    \[ x_i=\sem_1(\ldots,x_i)-\sem_1(\ldots,x_{i-1}). \]
    Thus, $f$ is a finite linear combination of products of the generators $\sem_r(\ldots,x_k)$ with $k\geq b$. Such a product need not be standard because an endpoint may be repeated. The straightening identity from Fomin--Gelfand--Postnikov \cite[3.3]{FGP}, shifted to integer endpoints, is
    \begin{align*}
    \sem_r(\ldots,x_k)\sem_s(\ldots,x_k)
    ={}&\sem_r(\ldots,x_{k+1})\sem_s(\ldots,x_k)\\
      &+\sum_{a\geq1}\sem_{r-a}(\ldots,x_{k+1})
                    \sem_{s+a}(\ldots,x_k)\\
      &-\sum_{a\geq1}\sem_{r-a}(\ldots,x_k)
                    \sem_{s+a}(\ldots,x_{k+1}),
    \end{align*}
    where terms with a negative degree vanish. Replacing a pair with the smallest repeated endpoint by the right-hand side is exactly the Fomin--Gelfand--Postnikov straightening algorithm. It does not introduce an endpoint below $b$, and their termination argument applies unchanged because the total degree is fixed. The result is a finite linear combination of back stable SEMs.
    Applying the same argument to every homogeneous component shows that $\bE$ spans $\Rg$. Therefore, $\bE$ forms a $\Q$-basis of $\Rg$.
\end{proof}

The proof gives a finite procedure for computing the expansion. If $f$ is homogeneous of degree $d$ and
\[ f\in\Lambda(\ldots,x_b)\otimes\Q[x_{b+1},\ldots,x_q], \]
choose any $p\leq b-d+1$, expand $\rho_p(f)$ in the ordinary SEM basis, and retain the same coefficients while restoring the original integer endpoints. The resulting expression is the back stable SEM expansion of $f$. In particular, the back stable coefficients of $\bschub_w$ may be obtained from a single sufficiently long finite truncation.

\begin{definition}
    For all $w \in S_\Z$, define the back stable SEM coefficients by
    \[ \bschub_w = \sum_{\alpha \in \CC} \overleftarrow{\kappa_{\alpha}^{w}} \cdot \bsem{\alpha}. \]
\end{definition}

The grading gives an immediate restriction on these coefficients.

\begin{corollary}\label{cor:homogeneous-support}
    If $\overleftarrow{\kappa^w_\alpha}\neq 0$, then $|\alpha|=\ell(w)$.
\end{corollary}

\begin{proof}
    The back stable Schubert polynomial $\bschub_w$ is homogeneous of degree $\ell(w)$, while $\bsem{\alpha}$ is homogeneous of degree $|\alpha|$. Since the back stable SEMs form a homogeneous basis, only indices of degree $\ell(w)$ can occur in the expansion of $\bschub_w$.
\end{proof}

The corresponding CHM basis theorem will follow in Section 4 from Dynkin reversal.

\subsection{Stabilization Theorem}

The shift equivariance of the back stable SEM coefficients allows us to recover the full ordinary SEM expansions from the back stable expansion. Let
\[ \pi_+:\Rg\longrightarrow\Q[x_1,x_2,\ldots] \]
be the specialization which applies the augmentation $\Lambda^-\to\Q$, sets $x_i=0$ in the polynomial factor for $i\leq0$, and fixes $x_i$ for $i>0$. If $u\in S_\infty$, then $\pi_+(\bschub_u)=\schub_u$. Under the same specialization, a back stable SEM $\bsem{\beta}$ becomes the ordinary SEM $\sem_\beta$ when $\beta$ is admissible, and vanishes otherwise.

\begin{definition}
    For $\alpha \in \CC$, we define its \emph{admissibility threshold} to be
    \[d(\alpha)=\max_{\alpha_i \neq 0} \{\alpha_i-i\},\]
    where $d(0)=0$. Then $\gamma^m(\alpha)$ is an admissible ordinary SEM index if and only if $m \geq d(\alpha)$.
\end{definition}

The following theorem formalizes this and allows us to recover finite SEM expansions from the back stable expansion.

\begin{theorem}[Stabilization]\label{thm:stabilization}
    Let $w \in S_\Z$, and let $m \ge 0$ obey $\gamma^m(w) \in S_\infty$.
    Then
    \[\schub_{\gamma^m(w)}=\sum_{\substack{\alpha \in \CC\\ d(\alpha)\leq m}}\overleftarrow{\kappa^w_\alpha} \cdot \sem_{\gamma^m(\alpha)}.\]
    Equivalently, for every $\alpha \in \CC$ satisfying $d(\alpha)\leq m$,
    \[\kappa^{\gamma^m(w)}_{\gamma^m(\alpha)} = \overleftarrow{\kappa^w_\alpha}.\]
\end{theorem}

\begin{proof}
    The shift satisfies
    \[ \gamma^m(\bschub_w)=\bschub_{\gamma^m(w)}, \qquad
       \gamma^m(\bsem{\alpha})=\bsem{\gamma^m(\alpha)}. \]
    Applying $\gamma^m$ to the back stable SEM expansion of $\bschub_w$ therefore gives
    \[ \overleftarrow{\schub}_{\gamma^m(w)}=\sum_{\alpha\in\CC}\overleftarrow{\kappa^w_\alpha} \cdot \overleftarrow{\sem}_{\gamma^m(\alpha)}.\]
    Now specialize $x_i=0$ for all $i\leq 0$. Since $\gamma^m(w)\in S_\infty$, the left-hand side becomes $\schub_{\gamma^m(w)}$. On the right,
    $\overleftarrow{\sem}_{\gamma^m(\alpha)}$ becomes
    $\sem_{\gamma^m(\alpha)}$ if $d(\alpha)\leq m$, and vanishes otherwise. Thus,
    \[\schub_{\gamma^m(w)}=\sum_{\substack{\alpha\in\CC \\ d(\alpha)\leq m}}\overleftarrow{\kappa^w_\alpha} \cdot \sem_{\gamma^m(\alpha)}.\]
    Since the SEMs form a basis, we have
    \[\kappa^{\gamma^m(w)}_{\gamma^m(\alpha)}=\overleftarrow{\kappa^w_\alpha}\]
    whenever $d(\alpha)\leq m$.
\end{proof}

\begin{corollary}
    For every $w\in S_\Z$, the shifted SEM expansion of $\schub_{\gamma^m(w)}$ has constant support and coefficients for all sufficiently large $m$.
\end{corollary}

\begin{proof}
    Only finitely many coefficients $\overleftarrow{\kappa^w_\alpha}$ are nonzero. Choose $m_0$ so that $\gamma^m(w)\in S_\infty$ and $m\ge d(\alpha)$ for every such $\alpha$ whenever $m\ge m_0$. Theorem~\ref{thm:stabilization} then gives
    \[ \schub_{\gamma^m(w)}=
       \sum_{\alpha\in\supp_{\bE}(\bschub_w)}
       \overleftarrow{\kappa^w_\alpha}\,\sem_{\gamma^m(\alpha)} \]
    for every $m\ge m_0$.
\end{proof}

Combined with the finite procedure above, the threshold in this corollary works. Compute the finite support $A$ from one sufficiently long truncation and take $m_0$ large enough so that $\gamma^{m_0}(w)\in S_\infty$ and
\[ m_0\geq\max_{\alpha\in A}d(\alpha). \]

The back stable expansion of $\bschub_w$ encodes all information about the ordinary SEM expansions of $\schub_{1^t \times w}$. When $w\in S_\infty$, this includes the expansion of $\schub_w$ itself: it suffices to retain the coefficients $\overleftarrow{\kappa^w_\alpha}$ with $d(\alpha)\le 0$. This is a principal advantage of passing into the back stable setting.

\section{Dynkin Reversal}

\subsection{The Reversing Map}

We now turn our attention to the symmetry between back stable SEMs and CHMs. SEMs and CHMs in the finite setting behave in seemingly unrelated ways. The symmetry arises from Dynkin reversal in the back stable setting. Consider the classical $\Q$-algebra involution $\omega : \Lambda \to \Lambda$ defined by $\omega(e_r) = h_r$ for any $r \ge 1$. Following Lam--Lee--Shimozono, we extend this involution to a $\Q$-algebra involution $\widetilde{\omega}: \Rg \to \Rg$.

\begin{definition}
    $\widetilde{\omega}: \Rg \to \Rg$ obeys
    \[ \widetilde{\omega}(e_r)=h_r, \qquad \widetilde{\omega}(x_i) = -x_{1-i}. \]
\end{definition}

There is a corresponding involution of $S_\Z$ given on simple reflections.

\begin{definition}
    For any $w = s_{a_1} s_{a_2} \cdots s_{a_{\ell(w)}}  \in S_\Z$, let $\widehat{w} = s_{-a_1} s_{-a_2} \cdots s_{-a_{\ell(w)}}$ be its \emph{Dynkin reversal}. Note that this definition is independent of the choice of reduced word and gives an involution of $S_\Z$. In other words,
    \[ w(i)=1-\widehat w(1-i). \]
    Moreover, for any $w \in S_n$ we define the reverse complement $w^{\mathrm{rc}}$ by
    \[ w^{\mathrm{rc}} = \gamma^n(\widehat{w}). \]
    We choose this convention because $w \in S_n$ implies $w^{\mathrm{rc}} \in S_n$; thus, reverse complement induces an involution on $S_n$ for every $n \ge 1$.

    For $\alpha \in \CC$, we let $\widehat{\alpha}_i = \alpha_{-i}$ for all $i$. This will be the analogous reversal map on $\CC$.
\end{definition}

Lam--Lee--Shimozono prove that these two involutions on $S_\Z$ and $\Rg$ are compatible with back stable Schubert polynomials.

\begin{theorem}[Lam--Lee--Shimozono \cite{LLS}]
    For all $w \in S_\Z$, one has $\widetilde{\omega} \left ( \bschub_w \right ) = \bschub_{\widehat{w}}$.
\end{theorem}

In the back stable SEMs, the map behaves particularly nicely.

\begin{theorem}
    For all $\alpha \in \CC$, 
    \[ \widetilde{\omega} \left ( \bsem{\alpha} \right ) = \bchm{\widehat{\alpha}}, \qquad \widetilde{\omega} \left ( \bchm{\alpha} \right ) = \bsem{\widehat{\alpha}}. \]
\end{theorem}

\begin{proof}
    The key idea is to understand how $\widetilde{\omega}$ acts on a single back stable elementary symmetric function. We compare the generating functions for elementary and complete homogeneous symmetric functions and show that they satisfy the exact same recurrence after applying $\widetilde{\omega}$. The result then follows by applying this factorwise and reindexing.
    We begin by proving the first equality. Since $\widetilde{\omega}$ is an algebra automorphism, we can distribute
    \[ \widetilde{\omega} \left ( \bsem{\alpha} \right ) = \prod_{i \in \Z} \widetilde{\omega} \left ( \sem_{\alpha_i}(\ldots, x_{i-1},x_i) \right ).\]
    Now, we extend $\widetilde{\omega}$ to $\Rg[[u]]$ coefficientwise and define
    \[ \EE_k (u) = \sum_{m \ge 0} \sem_m(\ldots, x_{k-1},x_k) \cdot u^m, \quad \HH_k(u) = \sum_{m \ge 0} \chm_m(\ldots, x_{k-1},x_k) \cdot u^m. \]
    The one-variable recurrences for elementary and complete homogeneous polynomials give, in $\Rg[[u]]$,
    \[ \EE_k(u) = (1 + x_k u) \cdot \EE_{k-1}(u), \qquad
       \HH_k(u) = (1 - x_k u)^{-1} \cdot \HH_{k-1}(u). \]
    We apply $\widetilde{\omega}$ to the first identity. $\widetilde{\omega}(x_k) = -x_{1-k}$ yields
    \[ \widetilde{\omega} \left ( \EE_k(u) \right ) = (1 - x_{1-k} u) \cdot \widetilde{\omega} \left ( \EE_{k-1}(u) \right ), \]
    and we replace $k$ with $1-k$ in the second identity and get
    \[ \HH_{-k}(u) = (1 - x_{1-k} u) \cdot \HH_{-(k-1)}(u). \]
    Thus the sequences $\left ( \widetilde{\omega}(\EE_k(u)) \right )_{k \in \Z}$ and
    $\left ( \HH_{-k}(u) \right )_{k \in \Z}$ satisfy the same first-order recursion, so agreement at one $k$ propagates to every $k \in \Z$. Thus, it suffices to consider $k = 0$. Here, every coefficient lies in $\Lambda$, where
    $\widetilde{\omega}$ restricts to the classical involution $\sem_m \mapsto \chm_m$, so
    $\widetilde{\omega}(\EE_0(u)) = \HH_0(u)$. Hence $\widetilde{\omega}(\EE_k(u)) = \HH_{-k}(u)$ for all $k \in \Z$,
    and comparing coefficients of $u^m$ gives
    \[ \widetilde{\omega} \left ( \sem_m(\ldots, x_{k-1}, x_k) \right ) = \chm_m(\ldots, x_{-k-1}, x_{-k}). \]
    Thus, returning to the product, we can reindex $j = -i$ to get
    $\widehat{\alpha}_j = \alpha_{-j}$,
    \[ \widetilde{\omega} \left ( \bsem{\alpha} \right )
       = \prod_{i \in \Z} \chm_{\alpha_i}(\ldots, x_{-i-1}, x_{-i})
       = \prod_{j \in \Z} \chm_{\widehat{\alpha}_j}(\ldots, x_{j-1}, x_j)
       = \bchm{\widehat{\alpha}}. \]
    The asserted CHM identity follows by applying $\widetilde{\omega}$ again, since both $\widetilde{\omega}$ and $\alpha\mapsto\widehat\alpha$ are involutions.
\end{proof}

\begin{corollary}
    The set $\bH$ forms a $\Q$-basis of $\Rg$.
\end{corollary}

\begin{proof}
    The automorphism $\widetilde{\omega}$ sends the basis $\bE$ bijectively onto $\bH$.
\end{proof}

\subsection{SEM and CHM Duality}

The map $\widetilde{\omega}$ acts nicely on back stable SEMs, CHMs, and Schubert polynomials. Using this structure, we can prove the following theorem:

\begin{theorem}
    For any $w \in S_\Z$ and $\alpha \in \CC$, the coefficient of $\bsem{\alpha}$ in $\bschub_w$ is equal to the coefficient of $\bchm{\widehat{\alpha}}$ in $\bschub_{\widehat{w}}$. Equivalently,
    \[ \bschub_{\widehat{w}} = \sum_{\alpha \in \CC} \overleftarrow{\kappa^w_\alpha} \cdot \bchm{\widehat{\alpha}}. \]
\end{theorem}

\begin{proof}
    Apply $\widetilde{\omega}$ to the back stable SEM expansion:
    \begin{align*}
        \bschub_{\widehat{w}} &= \widetilde{\omega} \left ( \bschub_{w} \right ) \\
        &= \sum_{\alpha \in \CC} \overleftarrow{\kappa^w_\alpha} \cdot \widetilde{\omega}\left ( \bsem{\alpha} \right ) \\
        &= \sum_{\alpha \in \CC} \overleftarrow{\kappa^w_\alpha} \cdot \bchm{\widehat{\alpha}}.
    \end{align*}
\end{proof}

\begin{corollary}
    \[ \bschub_w = \sum_{\alpha \in \CC} \overleftarrow{\kappa^w_\alpha} \cdot \bsem{\alpha} = \sum_{\alpha \in \CC} \overleftarrow{\kappa^{\widehat w}_{\widehat \alpha}} \cdot \bchm{\alpha}. \]
\end{corollary}

If $\alpha, \beta_1, \beta_2 \in \CC$ obey $\alpha = \beta_1+\beta_2$ and $w, u_1, u_2 \in S_\Z$ obey $w = u_1u_2$, then

\begin{itemize}
    \item $\widehat{\alpha} = \widehat{\beta_1} + \widehat{\beta_2}$ and $\widehat{w} = \widehat{u_1} \widehat{u_2}$;
    \item $|\alpha| = |\widehat{\alpha}|$ and $\ell(w) = \ell(\widehat{w})$;
    \item If $w$ is $u$-avoiding for $u \in S_\infty$, then $\widehat{w}$ is $u^{\mathrm{rc}}$-avoiding.
\end{itemize}

The implications here are significant. Since the reversing map respects the relevant operations on $\Rg$, $S_\Z$, and $\CC$, any result regarding back stable SEM expansions expressed in this language has a corresponding dual result for back stable CHMs. Moving forward, we will primarily emphasize results concerning back stable SEMs, dualizing them to CHMs when the resulting statement is worth mentioning.

\section{Orbit Sums}

\subsection{Stanley Symmetric Functions}

\begin{definition}
    For $w \in S_\infty$, define the \emph{Stanley symmetric function} of $w$ to be
    \[ F_w = \lim_{m \to \infty} \schub_{1^m \times w} \in \Lambda. \]
    The limit is taken coefficientwise. We extend this definition to $S_\Z$ by requiring $F_{\gamma(w)} =F_w$ for all $w \in S_\Z$. This is well defined because for every $w \in S_\Z$, there exists $k$ such that $\gamma^k(w) \in S_\infty$, and thus $F_w = F_{\gamma^k(w)}$.
\end{definition}

\begin{definition}
    The group $S_\Z$ naturally acts on $\CC$ by permuting indices. For $\alpha\in\CC$, let $\Orb(\alpha)$ denote its orbit. Each orbit is associated with a unique partition, obtained by sorting the nonzero entries of $\alpha$ in weakly decreasing order. This gives the natural identification
    \[ \CC/S_\Z \cong \Par. \]
    Accordingly, for $\lambda \in \Par$, we abuse notation by writing $\Orb(\lambda)$ for the corresponding set of indices in $\CC$.
\end{definition}

The Stanley symmetric functions are naturally compatible with back stable SEMs. Indeed, the SEM coefficients of $\schub_{1^m \times w}$ after shifting stabilize to the back stable coefficients $\overleftarrow{\kappa^w_\alpha}$. Consequently, for $w \in S_\infty$
\begin{align*}
    F_w &= \lim_{m \to \infty} \schub_{1^m \times w} \\
    &= \sum_{\alpha \in \CC} \overleftarrow{\kappa_\alpha^w} \cdot \lim_{m \to \infty} \sem_{\gamma^m(\alpha)}.
\end{align*} 
\begin{proposition}
    For any $\alpha \in \CC$,
    \[ \lim_{m \to \infty} \sem_{\gamma^m(\alpha)} = e_{\parz(\alpha)} \in \Lambda, \]
    where $e_\lambda$ denotes the elementary symmetric function associated with $\lambda \in \Par$. In particular, the limit $\lim_{m \to \infty} \sem_{\gamma^m(\alpha)}$ depends only on $\Orb(\alpha)$.
\end{proposition}
\begin{proof}
    We may compute
    \begin{align*}
    \lim_{m \to \infty} \sem_{\gamma^m(\alpha)} &= \prod_{i \in \Z} \lim_{m \to \infty} \sem_{\alpha_i}(x_1, \ldots, x_{i+m}) \\
    &= \prod_{i \in \Z} e_{\alpha_i}(x_1,x_2, \ldots ) \\
    &= e_{\parz(\alpha)}.
    \end{align*}
\end{proof}
The CHM analogue follows similarly.
\[ \lim_{m \to \infty} \chm_{\gamma^m(\alpha)} = h_{\parz(\alpha)} \in \Lambda, \]
where $h_{\parz(\alpha)}$ is the complete homogeneous symmetric function associated with $\parz(\alpha) \in \Par$. Now that the limit of $\sem_{\gamma^m(\alpha)}$ is well understood, we can finish our computation of $F_w$.
\begin{align*}
    F_w &= \sum_{\alpha \in \CC} \overleftarrow{\kappa^w_\alpha} \cdot \lim_{m \to \infty} \sem_{\gamma^m(\alpha)} \\
    &= \sum_{\alpha \in \CC}  \overleftarrow{\kappa^w_\alpha} \cdot e_{\parz(\alpha)} \\
    &= \sum_{\lambda \in \Par} \left ( \sum_{\alpha \in \Orb(\lambda)} \overleftarrow{\kappa^w_\alpha} \right ) \cdot e_\lambda.
\end{align*}
This is a unique expansion, as $\{ e_\lambda : \lambda \in \Par \}$ forms a basis. The same reasoning applies to back stable CHMs and gives the following theorem regarding back stable SEM and CHM coefficients.
\begin{theorem}
    For $w \in S_\Z$ and $\lambda \in \Par$,
    \[ \sum_{\alpha \in \Orb(\lambda)} \overleftarrow{\kappa^w_\alpha} = [e_\lambda] F_w, \qquad \sum_{\alpha \in \Orb(\lambda)} \overleftarrow{\kappa^{\widehat{w}}_{\widehat{\alpha}}} = [h_\lambda] F_w. \]
\end{theorem}

This prompts us to regard the back stable SEM coefficients as position-sensitive refinements of the elementary symmetric coefficients of the Stanley symmetric function $F_w$. While $F_w$ only remembers the partition $\lambda$, the back stable SEM basis distinguishes individual compositions $\alpha$ with shape $\lambda$.

\subsection{Total Sum of Back Stable SEM Coefficients}

We can collapse the preceding orbit-sum identity further by summing over all possible orbits. The result is particularly nice and rigid.

\begin{lemma}\label{lem:total-coefficient-sum}
    For every $w\in S_\Z$,
    \[ \sum_{\alpha \in \CC} \overleftarrow{\kappa^w_\alpha} = \begin{cases}
        1 & \text{if $\bschub_w$ is a single back stable SEM;} \\
        0 & \text{otherwise.}
    \end{cases} \]
\end{lemma}

\begin{proof}
    Both the coefficient sum and the condition of being a single back stable SEM are invariant under shifting $w$. We may therefore assume that $w\in S_\infty$. Now
    \[\sum_{\alpha \in \CC} \overleftarrow{\kappa^w_\alpha}=\sum_{\lambda \in \Par}\sum_{\alpha \in \Orb(\lambda)}\overleftarrow{\kappa^w_\alpha}=\sum_{\lambda \in \Par} [e_\lambda] F_w.\]
    Thus, if we let $\chi : \Lambda \to \Z$ be the algebra map satisfying
    \[\chi\left(e_r\right)=1\qquad\text{for }r \ge 1,\]
    then our desired value is simply $\chi(F_w)$. Expand
    \[F_w=\sum_{\lambda\vdash \ell(w)} a_{\lambda,w}s_\lambda.\]
    By the Edelman--Greene rule \cite{EG}, $a_{\lambda,w}$ is the number of row-and-column-strict tableaux of shape $\lambda$ whose row-reading word is a reduced word for $w$. By the second Jacobi--Trudi identity,
    \[\chi(s_\lambda)=\chi\left(\det\left(e_{\lambda'_i+j-i}\right)_{i,j=1}^{\ell(\lambda')}\right).\]
    The map $\chi$ sends each $e_r$ to $1$ if $r\geq 0$, where $e_0=1$, and to $0$ otherwise.
    If $\lambda_1\geq 2$, then the first two rows of the image are all $1$'s,
    and thus the matrix is singular. Hence this case vanishes, and
    $\lambda_1=1$ is the only case we have to consider. This is equivalent to
    $\lambda=1^{\ell(w)}$, for which $\chi(s_\lambda)=1$. Therefore
    \[\chi(F_w)=a_{1^{\ell(w)},w}.\]
    This coefficient is the number of strictly increasing reduced words of $w$.

    \begin{claim}
        The permutation $w$ has a strictly increasing reduced word if and only if it avoids $321$ and $312$. Moreover, such a reduced word is unique.
    \end{claim}

    \begin{proof}
        Induct on $n$ for $w\in S_n$. If $w$ avoids $321$ and $312$, then $n$ must lie in the last or penultimate position. Deleting $n$ preserves avoidance, and by induction the resulting permutation has a unique strictly increasing reduced word; according to the position of $n$, the reduced word for $w$ is either unchanged or obtained by appending $n-1$. Conversely, in any strictly increasing reduced word, the letter $n-1$, if present, must occur last, so removing it reduces to the case of $S_{n-1}$. This gives the avoidance characterization and uniqueness.
    \end{proof}

    We may use this exact pattern avoidance condition in a related claim.

    \begin{claim}
        $\bschub_w$ is a single back-stable SEM if and only if $w$ is $321$ and $312$ avoiding.
    \end{claim}

    \begin{proof}
        Note that $\bschub_w$ being a single back stable SEM is equivalent to $\schub_{1^m \times w}$ being a single SEM for all large $m$, as the limit is coefficientwise. By \cite{Woodruff} this only occurs when $1^m \times w$ is $312$ and $1432$-avoiding. Since $1^m \times w$ has many left adjoined fixed points, $1432$-avoidance is equivalent to $321$-avoidance. This gives $321$ and $312$ as a necessary and sufficient condition for $\bschub_w$ being a single back-stable SEM.
    \end{proof}

    Putting these two claims together, we have the desired value $a_{1^{\ell(w)},w}$ being $1$ if and only if $w$ is $321$ and $312$ avoiding, which is equivalent to $\bschub_w$ being a single SEM. Conversely, if $\bschub_w$ is not a single SEM, then $w$ is not $321$ and $312$ avoiding. This is equivalent to $a_{1^{\ell(w)},w}=0$. 
\end{proof}

Dynkin reversal gives the corresponding statement for complete homogeneous monomials.

\begin{corollary}
    The sum of the coefficients in the back stable CHM expansion of $\bschub_w$ is $1$ if $\bschub_w$ is a single back stable CHM, and is $0$ otherwise.

    For every $w\in S_\Z$,
    \[ \sum_{\alpha \in \CC} \overleftarrow{\kappa_{\alpha}^{\widehat{w}}} = \begin{cases}
        1 & \text{if $\bschub_w$ is a single back stable CHM;} \\
        0 & \text{otherwise.}
    \end{cases} \]
\end{corollary}

The proof of Lemma~\ref{lem:total-coefficient-sum} shows that $\bschub_w$ is a single back stable SEM if and only if $w$ avoids $321$ and $312$. By Dynkin duality,
\[ \bschub_w = \bchm{\alpha} \Longleftrightarrow \bschub_{\widehat{w}} = \bsem{\widehat{\alpha}}. \]
Thus, $\bschub_w$ is a single back stable CHM if and only if $w$ avoids $321^{\mathrm{rc}} = 321$ and $312^{\mathrm{rc}} = 231$. This agrees with the ordinary CHM case described by Woodruff \cite[1.3]{Woodruff}.

\section{Shifted Specialization Polynomials}

\subsection{Polynomiality}

We now investigate the stabilization of $\schub_{1^t \times w}$ under principal specialization at $1$. For $t \ge 0$, consider
\[\schub_{1^t \times w}(1,1,\ldots,1),\]
which we abbreviate to $\schub_{1^t \times w}(1)$. This study builds on Macdonald's reduced-word identity and on the work of Fomin and Kirillov \cite{FK}, who considered the corresponding shifted reduced-word polynomial and, for dominant permutations, related it to plane partitions with bounded parts. Principal specializations of Schubert polynomials have also been studied in \cite{StanleyReciprocity,Gao,NadeauTewari,NadeauTewariForest}. Our contribution here is the back stable SEM formulation and a reduced word identity which connects the back stable SEM coefficients to reduced words.

Recall that the pipe dream formula expresses $\schub_u$ as a sum of monomials indexed by reduced pipe dreams. It follows that $\schub_u(1)$ is the number of reduced pipe dreams of $u$. Under the embedding $u\mapsto 1^t\times u$, the pipe dream diagram acquires $t$ additional rows above the original diagram, providing new positions in which crosses may occur. 

By the Macdonald reduced word identity \cite[6.11]{Macdonald}, building on the study of reduced decompositions by Stanley \cite{Stanley},
\[ \schub_{u}(1) = \frac{1}{\ell(u)!}\sum_{a \in R(u)} a_1 a_2 \cdots a_{\ell(u)}. \]
Here $R(u)$ is the set of reduced words of $u$. Using this expression, we can now investigate $\schub_{\gamma^t(w)}$ as a function of $t$. 

\begin{definition}
For any $w \in S_\Z$, define
\[ \T_w(t) = \frac{1}{\ell(w)!}\sum_{a \in R(w)} (t+a_1)(t+a_2) \cdots (t+a_{\ell(w)}) \]
to be the \emph{shifted specialization polynomial} of $w$.
\end{definition}

Note that if $w \in S_\infty$ and $t \ge 0$ then
\[ s_{a_1}s_{a_2} \cdots s_{a_{\ell(w)}} =w \Longleftrightarrow s_{a_1+t}s_{a_2+t} \cdots s_{a_{\ell(w)}+t} = \gamma^t(w). \]
which leads to the alternate formulation
\[ \T_w(t) = \frac{1}{\ell(w)!} \sum_{a \in R(\gamma^t(w))} a_1 a_2 \cdots a_{\ell(w)}. \]
Note that by the Macdonald formula, for $w \in S_\infty$ and $t \ge 0$

\begin{align*}
    \T_w(t) &= \frac{1}{\ell(w)!}\sum_{a \in R(1^t \times w)} a_1a_2 \ldots a_{\ell(w)} \\
    &= \schub_{1^t \times w}(1) \\
    &= |\RP(1^t \times w)|.
\end{align*}

Thus, $\T_w$ counts the number of reduced pipe dreams of $1^t \times w$. This can be viewed as adding $t$ empty rows above, which are now available for crosses to occupy. The polynomial nature of $\T_w$ tells us that adding $t$ upper rows admits polynomially many additional reduced pipe dreams. Similar to back stable SEM coefficients, the $\T_w$'s respect sliding relations.

\begin{fact}
$\T_{\gamma(w)}(t) = \T_w(t+1)$ for any $w \in S_\Z$.
\end{fact}

\begin{proof}
    Shifting a reduced word of $w$ increases each of its entries by one. Substituting this bijection into the definition gives the identity.
\end{proof}

Thus, up to this sliding relation, it suffices to consider $\T_w$ for $w \in S_\infty$.

\subsection{Back Stable SEM Formulation}

We can give another formulation of $\T_w$ using the $\overleftarrow{\kappa^w_\alpha}$ coefficients.

\begin{proposition}
    For $\alpha \in \CC$, let
    \[ B_\alpha(t) = \prod_{i \in \Z} \binom{t+i}{\alpha_i}. \]
    This is a polynomial in $t$, and $B_\alpha(t)=\sem_{\gamma^t(\alpha)}(1)$ whenever $t\ge d(\alpha)$. Then for any $w \in S_\Z$, one has
    \[ \T_w(t) = \sum_{\alpha \in \CC} \overleftarrow{\kappa^w_\alpha} \cdot B_\alpha(t). \]
\end{proposition}

\begin{proof}
    Only finitely many coefficients $\overleftarrow{\kappa^w_\alpha}$ are nonzero. Choose $t$ large enough that $\gamma^t(w)\in S_\infty$ and $t\ge d(\alpha)$ for every index in the support of $\bschub_w$. Theorem~\ref{thm:stabilization} gives
    \[ \schub_{\gamma^t(w)} = \sum_{\alpha \in \CC} \overleftarrow{\kappa^w_\alpha} \cdot \sem_{\gamma^t(\alpha)}. \]
    Taking the principal specialization at $1$ yields
    \[ \T_w(t) = \sum_{\alpha \in \CC} \overleftarrow{\kappa^w_\alpha} \cdot \sem_{\gamma^t(\alpha)}(1) = \sum_{\alpha \in \CC} \overleftarrow{\kappa^w_\alpha} \cdot B_\alpha(t). \]
    This equality holds for all sufficiently large integers $t$. Both sides are polynomials, so it holds identically.
\end{proof}

\begin{example}
    Take $w=312$. Its unique reduced word is $(2,1)$, so
    \[\T_{312}(t)= \frac{1}{2!}(t+2)(t+1)= \binom{t+2}{2}.\]
    We may also recover this from the back stable SEM expansion
    \[\overleftarrow{\schub}_{312}
    =\overleftarrow e_{(1,1,0, \ldots)}-\overleftarrow e_{(0,2, 0, \ldots)},
    \]
    Proposition~6.3 gives
    \[\T_{312}(t)=\binom{t+1}{1}\binom{t+2}{1}-\binom{t+2}{2}.\]
    Hence
    \[\T_{312}(t)= (t+1)(t+2)-\frac{(t+1)(t+2)}{2}= \frac{(t+1)(t+2)}{2},\]
    agreeing with the reduced-word formula.
\end{example}

The two formulations of $\T_w$ allow us to explicitly connect back stable SEM coefficients of $\bschub_w$ to the reduced words of $w$.

For $\alpha\in\CC$, write
\[ \binom{|\alpha|}{\alpha}=
   \frac{|\alpha|!}{\prod_{i\in\Z}\alpha_i!}. \]

\begin{lemma}
For any $w \in S_\infty$,
    \[ \sum_{\alpha \in \CC} \overleftarrow{\kappa^w_\alpha} \cdot \binom{|\alpha|}{\alpha} = |R(w)|. \]
\end{lemma}

\begin{proof}
    Consider the equality
    \[ \sum_{\alpha \in \CC} \overleftarrow{\kappa^w_\alpha} \cdot B_\alpha(t) = \frac{1}{\ell(w)!}\sum_{a \in R(w)} (t+a_1)(t+a_2) \cdots (t+a_{\ell(w)}). \]
    By Corollary~\ref{cor:homogeneous-support}, if $\overleftarrow{\kappa^w_\alpha} \neq 0$ then $|\alpha| = \ell(w)$, and thus each binomial product has degree $\ell(w)$. The leading coefficient of the left-hand side is therefore
    \[ \sum_{\alpha \in \CC} \frac{\overleftarrow{\kappa^w_\alpha}}{\prod_{i \in \Z} \alpha_i!}. \]
    Moreover, the leading coefficient of the right-hand side is $|R(w)|/\ell(w)!$. Hence,
    \[ |R(w)| = \ell(w)! \cdot \sum_{\alpha \in \CC} \frac{\overleftarrow{\kappa^w_\alpha}}{\prod_{i \in \Z} \alpha_i!} = \sum_{\alpha \in \CC} \overleftarrow{\kappa^w_\alpha} \cdot \binom{|\alpha|}{\alpha}. \]
\end{proof}

\section{Worked Examples}

We conclude by looking at families of permutations which have particularly nice back stable SEM expansions. Many of these permutations are known to be well behaved in the ordinary SEM expansion, but the back stable SEM expansions tend to reveal additional structure. 

\subsection{Longest Permutation}

Let $w^{(n)}=n(n-1)\cdots 21$ be the longest permutation in $S_n$. The flagged dual Jacobi--Trudi formula \cite{HJLM} gives, for every $m\geq 0$,
\[ \schub_{1^m\times w^{(n)}}
   =\det\left(
       \sem_{2j-i}(x_1,\ldots,x_{m+j})
     \right)_{1\leq i,j\leq n-1}. \]
Here and below, $\sem_0=1$ and $\sem_d=0$ for $d<0$. Shifting the variables by $-m$ and taking the coefficientwise limit gives the back stable determinant
\[\bschub_{w^{(n)}}=\det\left(\sem_{2j-i}(\ldots,x_{j-1},x_j)\right)_{1\leq i,j\leq n-1}.\]
Every elementary symmetric function in a term of the Leibniz expansion has a different right endpoint, since that endpoint is fixed by its column. Thus, no SEM straightening is required. Let
\[ A_n=\{\pi\in S_{n-1}:\pi(k)\leq 2k\text{ for every }k\in[n-1]\}. \]
For $\pi\in A_n$, define $\alpha_\pi\in\CC$ by
\[ (\alpha_\pi)_k = 2k-\pi(k), \qquad k \in [n-1] \]
and set all other entries equal to zero. Expanding by columns and omitting the terms containing an elementary symmetric function of negative degree gives
\[ \bschub_{w^{(n)}} = \sum_{\pi \in A_n} \operatorname{sgn}(\pi) \cdot \bsem{\alpha_\pi}. \]
The map $\pi\mapsto\alpha_\pi$ is injective on the nonzero terms, since $\pi(k)=2k-(\alpha_\pi)_k$. Consequently,
\[ \overleftarrow{\kappa^{w^{(n)}}_\alpha}\in\{0,\pm1\} \]
for every $\alpha\in\CC$.

For the dual statement in back stable CHMs,
\[ \widehat{w^{(n)}} = \gamma^{-n} \left ( w^{(n)} \right ) \]
and thus Dynkin duality tells us
\[ \bschub_{w^{(n)}} = \sum_{\pi \in A_n} \operatorname{sgn}(\pi) \cdot \bchm{\alpha_\pi'} \]
where $\alpha'_\pi = \gamma^n(\widehat{\alpha_\pi})$. In particular,
\[ (\alpha'_\pi)_k = (\alpha_\pi)_{n-k} = 2n-2k-\pi(n-k). \]

\begin{lemma}
    \[ \left| \supp_{\bE} \left ( \bschub_{w^{(n)}} \right ) \right| = \left \lfloor \frac n2 \right \rfloor! \cdot \left \lceil \frac n2 \right \rceil! \]
\end{lemma}

\begin{proof}
    We simply compute the number of $\pi \in S_{n-1}$ satisfying the nonnegativity constraint. Equivalently, we require $\pi(k) \le 2k$ for every $k\in[n-1]$. Construct $\pi$ from left to right. When position $k$ is reached, every previously used value lies in $[\min(2k,n-1)]$, and therefore there are exactly
    \[ \min(2k,n-1)-(k-1) = \min(k+1, n-k) \]
    choices at this step. Multiplying over $k$ gives
    \[ \prod_{k=1}^{n-1} \min(k+1, n-k)
       =  \left \lfloor \frac n2 \right \rfloor! \cdot \left \lceil \frac n2 \right \rceil!. \]
\end{proof}

\subsection{Grassmannian Permutations}

Let $w\in S_\infty$ be Grassmannian. Write $k$ for its unique descent and $\lambda$ for its associated partition. Equivalently,
\[ \lambda=(L(w)_k,L(w)_{k-1},\ldots,L(w)_1) \]
after trailing zero parts are removed. The usual Grassmannian identity is
\[ \schub_w=s_\lambda(x_1,\ldots,x_k). \]
If $r=\lambda_1$ and $\lambda'$ denotes the conjugate partition, Winkel's variant of the dual Jacobi--Trudi formula \cite[2.1]{Winkel}, related to the flagged Schur framework of Wachs \cite{Wachs}, states that
\[ s_\lambda(x_1,\ldots,x_k)
   =\det\left(
       \sem_{\lambda'_i-i+j}(x_1,\ldots,x_{k+j-1})
     \right)_{1\leq i,j\leq r}. \]
Applying this identity to $1^m\times w$, shifting by $-m$, and taking the back stable limit yields the following formula.

\begin{proposition}\label{prop:grassmannian-back-stable-sem}
    If $w$ is Grassmannian of shape $\lambda$ with unique descent at $k$, then
    \[ \bschub_w
       =\det\left(
          \sem_{\lambda'_i-i+j}(\ldots,x_{k+j-2},x_{k+j-1})
        \right)_{1\leq i,j\leq \lambda_1}. \]
\end{proposition}

The determinant is already an alternating sum of back stable SEMs, since different columns have different right endpoints. More explicitly, let
\[ A_\lambda=\{\pi\in S_r:\lambda'_{\pi(j)}-\pi(j)+j\geq0\text{ for every }j\in[r]\}. \]
For $\pi\in A_\lambda$, define $\alpha^\pi\in\CC$ by
\[ \alpha^\pi_{k+j-1}=\lambda'_{\pi(j)}-\pi(j)+j,
   \qquad j\in[r], \]
and set its other entries equal to zero. Then
\[\bschub_w=\sum_{\pi\in A_\lambda}\operatorname{sgn}(\pi) \cdot \bsem{\alpha^\pi}.\]
Since the integers $\lambda'_i-i$ are strictly decreasing in $i$, the index $\alpha^\pi$ determines $\pi$. Hence every back stable SEM coefficient of a Grassmannian permutation belongs to $\{0,\pm1\}$.

\section{Further Directions}    

\subsection{Cancellation-Free Formula}

The coefficients $\overleftarrow{\kappa^w_\alpha}$ need not be positive, and the known determinantal formulas conceal substantial cancellation. It would be useful to find a direct signed model for these coefficients, together with a sign-reversing involution that isolates the surviving objects. Even a cancellation-free formula for a broad pattern-avoiding family would give a more conceptual explanation for the small coefficients observed in finite SEM expansions. The stabilization theorem suggests that such a model should be formulated directly in terms of $w$ and $\alpha$, rather than generated separately for every $1^m\times w$.

\end{document}